\documentclass[reqno]{amsart}
\usepackage{amsfonts}
\usepackage{amsmath}
\usepackage{amssymb}
\usepackage{hyperref}
\usepackage{cleveref}

\usepackage{textcomp}

\usepackage[figuresright]{rotating}

\usepackage{subcaption}
\usepackage{amsthm}
\usepackage[labelfont=bf,justification=raggedright,singlelinecheck=false]{caption}
\newtheorem{theorem}{Theorem}
\theoremstyle{plain}

\newtheorem{definition}{Definition}
\newtheorem{example}{Example}

\newtheorem{proposition}{Proposition}
\newtheorem{remark}{Remark}

\numberwithin{equation}{section}

\begin{document}
\title[Rectifying-type Curves and Rotation Minimizing Frame $\mathbb{R}^{n}$]{Rectifying-type Curves and Rotation Minimizing Frame $\mathbb{R}^{n}$}
\author{\"{O}zg\"{u}r Keskin*}
\address{\"{O}zg\"{u}r Keskin*: Ankara University, Faculty of Science,
	Department of Mathematics, 06100, Tando\u{g}an, Ankara, Turkey.}
\email{ozgur.keskin@ankara.edu.tr}
\author{Yusuf Yayl{\i}}
\address{Yusuf Yayl{\i}:Ankara University, Faculty of Science,
	Department of Mathematics, 06100, Tando\u{g}an, Ankara, Turkey.}
\email{yayli@science.ankara.edu.tr}


\keywords{Myller configuration; Rectifying-type curve; Rotation minimizing frame (RMF) \\ 
 \indent *Corresponding author. mail:ozgur.keskin.mat@gmail.com}

\begin{abstract}
\indent In this paper, we have first given easily the characterization of special curves with the help of the Rotation minimizing frame (RMF). Also, rectifying-type curves are generalized n-dimensional space $\mathbb{R}^{n}$.

\end{abstract}

\maketitle

\section{Introduction}

\indent Rotation minimizing frames (RMFs) are presented by Bishop as an alternative to the Frenet moving frame along a curve $\gamma$ in $\mathbb{R}^{n}$. The Frenet frame is an orthonormal frame which can be defined for curves in $\mathbb{R}^{n}$. As a result, both Frenet frame and RMF are orthonormal frames. A RMF along a curve $\beta$=$\beta(s)$ in $\mathbb{R}^{n}$ is defined by the tangent vector and normal vectors $N_{i}$ ($i=1,...,(n-1)$). $N'_{i}(s)$ is proportional to $\beta'(s)$. Such a normal vector field along a curve is called to be a Rotation minimizing vector field. Any orthonormal basis $\{\beta'(s_{0}), N_{1}(s_{0}), ..., N_{n-1}(s_{0}) \}$ at a point $\beta(s_{0})$ expresses a unique RMF along the curve $\gamma$. Hence, such a RMF is uniquely designated modula of a rotation in $\mathbb{R}^{n-1}$ \cite{etayo1, etayo2, jut, wang}.\\
\par Recently, RMF is largely used in computer graphics, including sweep or blending surface modeling, motion design and control in computer animation and robotics, etc. This issue has begun to attract attention among researchers. Let's briefly express some of them. A new ordinary and influential method for certain and steady computation of RMF of a curve in $3D$ is expressed in \cite{wang}. This method called the double reflection method uses two reflections to compute each frame from its preceding one to yield a sequence of frames to approach an exact RMF. Rotation minimizing frames of space curves are used for sweep surface modeling \cite{ etayo1, etayo2, jut}. Moreover, Legender curves on the unit tangent bundle are obtained by using the Rotation minimizing (RM) vector fields. Then, the ruled surfaces corresponding  to these Legender curves are given and the singularities of these ruled surfaces are investigated \cite{bekar}.\\
\par A versor field (i.e., unit vector field) $(C,\bar{\xi})$ or a plane field $(C,\pi)$ is examined. A pair $\{(C,\bar{\xi}),(C,\pi)\}$, $\bar{\xi} \in \pi$ is called a Myller configuration in $\mathbb{E}^{3}$ and is denoted by $\mathbb{M}(C,\bar{\xi},\pi )$. If, moreover, the planes $\pi$ are tangent to $C$, then we have a tangent Myller configuration $\mathbb{M}_{t}(C,\bar{\xi},\pi )$ \cite{miron}. The geometry of the vector field $(C,\bar{\xi})$ on a surface $S$ is the geometry of associated Myller configurations $\mathbb{M}_{t}(C,\bar{\xi},\pi )$. The geometric theory of $\mathbb{M}_{t}(C,\bar{\xi},\pi )$ represents a particular case of the general Myller configuration $\mathbb{M}(C,\bar{\xi},\pi )$. In the case when $\mathbb{M}_{t}(C,\bar{\xi},\pi )$ is associated Myller configuration to a curve $C$ on a surface $S$ one obtains the classical theory of curves on surfaces. This Myller configuration plays an important role in defining of rectifying-type curves. Rectifying-type curves are studied with the help of the Myller configuration given for three dimensional space in \cite{Mac,miron}.
\par In this paper, we give an implementation of RMF using Myller configuration and generalize here rectifying-type curves n-dimensional space $\mathbb{R}^{n}$. Also, we have seen that special curves are characterized very easily with the help of this frame and the hypothesis \lq \lq the derivative of the rectifying type curves  is of the rectifying-type curves\rq \rq\  makes an important contribution to the classification of special curves.

\section{Preliminaries}

\par Let $\alpha:I \subset \mathbb{R} \rightarrow \mathbb{R}^{n}$ be arbitrary curve in $\mathbb{R}^{n}$. $\alpha$ is said a unit speed (or parameterized by arclength function) curve if $\langle{\alpha'(s),\alpha'(s)}\rangle=1$. Also, the standard inner product of $\mathbb{R}^{n}$ is given by 
\begin{eqnarray} \nonumber
\langle{X,Y}\rangle=\sum_{i=1}^{n}{x_{i}y_{i}}
\end{eqnarray}
for each \  $X=(x_{1},x_{2},...,x_{n}), \ Y=(y_{1},y_{2},...,y_{n}) \ \in \  \mathbb{R}^{n}$. The norm of a vector $u \ \in \ \mathbb{R}^{n}$ is given by $\Arrowvert u \lVert ^{2}=\langle{u,u}\rangle$. 
\begin{definition}
	A normal vector field $\vec{V}=\vec{V}(t)$ over a curve $\gamma=\gamma(t)$ in $\mathbb{R}^{n}$ is said to be relatively parallel or RM if the derivative \ $\vec{V}'(t)$ is proportional to \ $\gamma'(t)$ \cite{etayo1,etayo2}.
\end{definition}
\begin{definition}
	Let $\gamma=\gamma(t)$ in $\mathbb{R}^{n}$ be a curve. A RMF, parallel frame, natural frame, Bishop frame or adapted frame is a moving orthonormal frame $\{\vec{T}(t), \vec{N_{i}}(t)\}$, $i=1,2,...,n-1$ along $\gamma$, where $\vec{T}(t)$ is the tangent vector to $\gamma$ at the point $\gamma(t)$ and $\vec{N_{i}}(t)=\{\vec{N_{1}}(t), \vec{N_{2}}(t),...,\vec{N_{n-1}}(t)\}$ are RM vector fields \cite{etayo1,etayo2}. 
\end{definition}
\par The formulae of the Rotation minimizing frame $\{{\xi}_{1},{\xi}_{2},{\xi}_{3},{\xi}_{4}\}$ on $\int {\xi}_{1} ds$ is given as follows \cite{keskin,keskin1}:
\begin{eqnarray} \label{rmf1}
\left[ 
\begin{array}{c}
{\xi}_{1}^{\prime }(s) \\ 
{\xi}_{2}^{\prime }(s) \\ 
{\xi}_{3}^{\prime }(s)\\
{\xi}_{4}^{\prime }(s)
\end{array}
\right] =\left[ 
\begin{array}{cccc}
0 & \overline{k}_{1}& \overline{k}_{2}& \overline{k}_{3}\\
-\overline{k}_{1} & 0 & 0 & 0 \\ 
-\overline{k}_{2} & 0 & 0 & 0\\
-\overline{k}_{3} & 0 & 0 & 0
\end{array}
\right] \left[ 
\begin{array}{c}
{\xi}_{1}(s) \\ 
{\xi}_{2}(s) \\ 
{\xi}_{3}(s)\\
{\xi}_{4}(s)
\end{array}
\right],  
\end{eqnarray}
where $\overline{k}_{1}$,\  $\overline{k}_{2}$ \ and \ $\overline{k}_{3}$ \ are Rotation minimizing curvatures of \ $\int {\xi}_{1} ds$ \ curve.
Also, the formulae of the generalization Rotation minimizing frame $\{{\xi}_{1},{\xi}_{2},...,{\xi}_{n}\}$ on $\int {\xi}_{1} ds$ is given as follows. Here, $\overline{k}_{1}$, $\overline{k}_{2}$,...,$\overline{k}_{n-1}$ are Rotation minimizing curvatures of $\int {\xi}_{1} ds$ curve \cite{keskin,keskin1}.
\begin{eqnarray}\label{rmf2}
\left[ 
\begin{array}{c}
{\xi}_{1}^{\prime }(s) \\ 
{\xi}_{2}^{\prime }(s) \\ 
{\xi}_{3}^{\prime }(s)\\
...\\
{\xi}_{n-1}^{\prime }(s)\\
{\xi}_{n}^{\prime }(s)
\end{array}
\right] =\left[ 
\begin{array}{ccccc}
0 & \overline{k}_{1}&...& \overline{k}_{n-2}&\overline{k}_{n-1}\\
-\overline{k}_{1} &0 &...& 0 & 0 \\ 
-\overline{k}_{2} &0 & ...& 0 & 0\\
... & ... & ... & ...&....\\
-\overline{k}_{n-2} &0 &...& 0 & 0\\
-\overline{k}_{n-1} & 0 & ... & 0 & 0
\end{array}
\right] \left[ 
\begin{array}{c}
{\xi}_{1}(s) \\ 
{\xi}_{2}(s) \\ 
{\xi}_{3}(s)\\
...\\
{\xi}_{n-1}(s)\\
{\xi}_{n}(s)
\end{array}
\right].  
\end{eqnarray}
\section{Rectifying-Type Curves and Rotation Minimizing Frame}
Rectifying curves are introduced by Chen in \cite{chen1} as space curves whose position vector always lies in its rectifying plane, spanned by the tangent and the binormal vector fields $\vec{T}$ and $\vec{B}$ of the curve. Therefore, the position vector $\vec{\alpha}$ of a rectifying curve satisfies the equation

\begin{eqnarray}\nonumber
\vec{\alpha}(s)=\lambda(s)\vec{T}(s)+\mu(s)\vec{B}(s),
\end{eqnarray}
for some differentiable functions \ $\lambda(s)=s+b$ and \ $\mu(s)=c \ \in \ R$ in arclength functions $s$. Moreover, Chen \cite{chen1} proved that a curve in $\mathbb{R}^{3}$ with $\kappa \textgreater 0$ is congruent to a rectifying curve if and only if the ratio $ \left( \dfrac{\tau}{\kappa}\right)  $ of the curve is a nonconstant linear function in arclength function $s$. The Euclidean rectifying curves are studied in \cite{chen1,chen2}.  \\
\begin{definition}
	$\bar{r}$ is a rectifying-type curve (or simply rectifying curve) in the Frenet-type frame $\mathcal{R_{F}}$ if
	\begin{eqnarray}
	\bar{r}=\lambda(s) \bar{\xi_{1}}(s)+\mu(s) \bar{\xi_{3}}(s),
	\end{eqnarray} 
	where $\lambda$, $\mu$ are functions \cite{Mac}. 
\end{definition}
\begin{theorem}
	Let $\bar{r}(s):I \rightarrow \boldsymbol{E}_{3}$ be a curve in $\boldsymbol{E}_{3}$ expressed in the Frenet-type frame $\mathcal{R_{F}}$ by:
	\begin{eqnarray}\nonumber
	\frac{d\bar{r}}{ds}(s)=a_{1}(s)\bar{\xi_{1}}(s)+a_{2}(s)\bar{\xi_{2}}(s)+a_{3}(s)\bar{\xi_{3}}(s),\end{eqnarray} with $K_{1}(s)	> 0$, such that one of the following items holds:
	\begin{itemize}
		\item[(i)] \begin{eqnarray}\nonumber
		\frac{d}{ds}(\langle{\bar{r}(s),\bar{\xi_{1}}(s)}\rangle)=a_{1}(s),
		\end{eqnarray}
		\item[(ii)] For $K_{2}(s)\neq 0$, 
		\begin{eqnarray*}
		\frac{d}{ds}(\langle{\bar{r}(s),\bar{\xi_{3}}(s)}\rangle)=a_{3}(s).
		\end{eqnarray*}
	\end{itemize}
	Then $\bar{r}(s)$ is a rectifying-type curve. Conversely, if $\bar{r}(s)$ is a rectifying-type curve, then i) and ii) hold \cite{Mac}.
\end{theorem}
\par Now, we will give a new characterization of rectifying-type curves \ using \ a Rotation minimizing frame in $\mathbb{R}^{3}$, $\mathbb{R}^{4}$, $\mathbb{R}^{n}$, respectively.
\section{Rectifying-Type Curves and Rotation Minimizing Frame in $\mathbb{R}^{3}$}
\par For $n=3$, the formulae of the Rotation minimizing frame $\{{\xi}_{1},{\xi}_{2},{\xi}_{3}\}$ on $\int {\xi}_{1} ds$ is given as follows:
\begin{eqnarray}
\left[ 
\begin{array}{c}
{\xi}_{1}^{\prime }(s) \\ 
{\xi}_{2}^{\prime }(s) \\ 
{\xi}_{3}^{\prime }(s)
\end{array}
\right] =\left[ 
\begin{array}{cccc}
0 & \overline{k}_{1}& \overline{k}_{2}\\
-\overline{k}_{1} & 0 & 0 \\ 
-\overline{k}_{2} & 0 & 0 
\end{array}
\right] \left[ 
\begin{array}{c}
{\xi}_{1}(s) \\ 
{\xi}_{2}(s) \\ 
{\xi}_{3}(s)
\end{array}
\right],  
\end{eqnarray}
where $\overline{k}_{1}$ and \ $\overline{k}_{2}$ \ are Rotation minimizing curvatures of \ $\int {\xi}_{1} ds$ \ curve \cite{keskin,keskin1}.\\
\begin{definition}
	Let \ $\beta(s)=\lambda(s) \xi_{2}(s)+\mu(s) \xi_{3}(s)$ \ be a rectifying-type curve. Then,
	\begin{itemize}
		\item[(i)]If \ $\mu(s)=constant=c$ is taken, $\beta(s)=\lambda(s) \xi_{2}(s)+c \xi_{3}(s)$ is defined as a type 1 rectifying-type curve.\\
		\item[(ii)] If  $\lambda(s)=constant=c$ is taken, $\beta(s)=c \xi_{2}(s)+\mu(s) \xi_{3}(s)$ is defined as a type 2 rectifying-type curve.\\
	\end{itemize}
\end{definition}
\begin{theorem}\label{thm}
	Let $\{\xi_{1},\xi_{2},\xi_{3}\}$ be a Rotation minimizing frame (i.e. Frenet-type frame) on $\alpha=\int \xi_{1}ds$ and 
	\begin{eqnarray}\nonumber
	\beta=\lambda(s)\xi_{2}+c\xi_{3}, \quad c \in \mathbb{R},
	\end{eqnarray}  be a type 1 rectifying-type curve. If $\beta'$ is a rectifying-type curve, the following items holds:
	\begin{itemize}
		\item[(i)] $ \lambda=-c \dfrac{\overline{k}_{2}}{\overline{k}_{1}} $, $ c \in \mathbb{R} $ and so, $ \beta $ is a rectifying curve. \\
		\item[(ii)] If \begin{eqnarray}\nonumber
		\lambda=-(c \frac{\overline{k}_{2}}{\overline{k}_{1}})=s+b, \quad c, b \in \mathbb{R}, \end{eqnarray} $\beta(s)=\int \xi_{2} ds $ is a rectifying curve.\\
		\item[(iii)] If \begin{eqnarray}\nonumber
		\lambda=-(c \frac{\overline{k}_{2}}{\overline{k}_{1}})=constant=c_{2}, \quad c_{2} \in \mathbb{R}, \end{eqnarray} $\beta(s)=U=c_{2}\xi_{2}+c\xi_{3} $ is a constant vector. Moreover, $\int \xi_{2} ds$ and $\int \xi_{3}ds$ curves are helix curves with $U$ axis.\\
		\item[(iv)] If \begin{eqnarray}\nonumber
		- \frac{\overline{k}_{2}}{\overline{k}_{1}}=\tan s, \end{eqnarray} \begin{eqnarray}\nonumber \beta(s)=c\sec s Y(s) \end{eqnarray} and $Y(s)\in S^{2}$. Hence, a type 1 rectifying-type curve $\beta$ is obtained from a spherical curve (i.e. from $Y(s)$ spherical curve). This situation is consistent with conditions in Chen's papers.\\
	\end{itemize}
	\end{theorem}
		\begin{proof}
		\begin{itemize}
			\item[(i)] Since \begin{eqnarray}\nonumber
			\beta'=\lambda'\xi_{2}+(-\lambda \overline{k}_{1}-c \overline{k}_{2})\xi_{1},
			\end{eqnarray} $\beta'$ is a rectifying-type curve if and only if \begin{eqnarray}\nonumber
			-\lambda \overline{k}_{1}-c \overline{k}_{2}=0 \end{eqnarray} and \begin{eqnarray}\nonumber
			\lambda=-(c \frac{\overline{k}_{2}}{\overline{k}_{1}}). \end{eqnarray}\\
			Thus, for
			 \begin{eqnarray}\nonumber
			\beta=-(c \frac{\overline{k}_{2}}{\overline{k}_{1}})\xi_{2}+c\xi_{3},\quad c \in \mathbb{R}, \end{eqnarray} curve,  $T_{\beta}=\xi_{2}$ and $N_{\beta}=\xi_{1}$. Resultly, $\langle{\beta,\xi_{1}}\rangle=0$ or $\langle{\beta,N_{\beta}}\rangle=0$ is obtained. $\beta$ is a rectifying curve. 
			
			\item[(ii)] For \begin{eqnarray}\nonumber
			-(c \frac{\overline{k}_{2}}{\overline{k}_{1}})=s+b, \quad c, b \in \mathbb{R}, \end{eqnarray} 
			\begin{eqnarray}\nonumber
			\beta=(s+b)\xi_{2}+c\xi_{3},\quad b,c \in \mathbb{R}, \end{eqnarray} rectifying curve overlap with $\int \xi_{2} ds $ curve.\\
			\item[(iii)] For \begin{eqnarray}\nonumber
			-(c \frac{\overline{k}_{2}}{\overline{k}_{1}})=constant=c_{2}, \quad c_{2} \in \mathbb{R}, \end{eqnarray} $U=c_{2}\xi_{2}+c\xi_{3} $ is obtained. Since \begin{eqnarray}\nonumber
			-c_{2} \overline{k}_{1}-c \overline{k}_{2}=0, \end{eqnarray} \begin{eqnarray}\nonumber
			\frac{dU}{ds}=(-c_{2} \overline{k}_{1}-c \overline{k}_{2})\xi_{1}=0,\end{eqnarray} is found. Moreover, tangents of $\int \xi_{2} ds$ and $\int \xi_{3}ds$ curves makes the fixed angle with $U$ axis. Then, $\int \xi_{2} ds$ and $\int \xi_{3}ds$ curves are helix curves with $U$ axis.\\
			\item[(iv)] For \begin{eqnarray}\nonumber
			-\frac{\overline{k}_{2}}{\overline{k}_{1}}=\tan s, \end{eqnarray} \begin{eqnarray}\nonumber \beta(s)=c(\tan s\xi_{2}+\xi_{3})=\frac{c}{\cos s}(\sin s\xi_{2}+\cos s\xi_{3})=c\sec s Y(s) \end{eqnarray} and $Y(s)\in S^{2}$.\\
		\end{itemize}
		\end{proof}

\begin{proposition}
	If $a\overline{k}_{1}+b\overline{k}_{2}+1=0, \quad  a,b \in \mathbb{R}$, $\int \xi_{1} ds$ is a spherical curve. Really, if $\int \xi_{1} ds$ is a spherical curve, $\int \xi_{1} ds$ is written as $\int \xi_{1} ds=a\xi_{2}+b\xi_{3}$. If derivative of both sides of equality is taken, $\xi_{1}=(-a\overline{k}_{1}-b\overline{k}_{2})\xi_{1}$ is obtained. Thus, $a\overline{k}_{1}+b\overline{k}_{2}+1=0, \quad  a,b \in \mathbb{R}$ is found. Also, $\Arrowvert \int \xi_{1} ds \lVert^{2}=r^{2}=a^{2}+b^{2}$. $r$ is radius of a sphere. 
\end{proposition}
We can write similar results in for type 2 rectifying-type curves.
\begin{theorem}
	Let $\{\xi_{1},\xi_{2},\xi_{3}\}$ be a Rotation minimizing frame (i.e. Frenet-type frame) on $\alpha=\int \xi_{1}ds$ and 
	\begin{eqnarray}\nonumber
	\gamma=c\xi_{2}+\mu(s) \xi_{3}, \quad c \in \mathbb{R},
	\end{eqnarray} be a type 2 rectifying-type curve. If $\gamma'$ is a rectifying-type curve, the following items holds:
	\begin{itemize}
		\item[(i)] \begin{eqnarray}\nonumber
		\gamma=c\xi_{2}-(c \frac{\overline{k}_{1}}{\overline{k}_{2}})\xi_{3},\quad c \in \mathbb{R}, \end{eqnarray} is a rectifying curve.\\
		\item[(ii)] If \begin{eqnarray}\nonumber
		-(c \frac{\overline{k}_{1}}{\overline{k}_{2}})=s+b, \quad c, b \in \mathbb{R}, \end{eqnarray} $\gamma(s)=\int \xi_{3} ds $ is a rectifying curve.\\
		\item[(iii)] If \begin{eqnarray}\nonumber
		-(c \frac{\overline{k}_{1}}{\overline{k}_{2}})=constant=c_{3}, \quad c_{3} \in \mathbb{R}, \end{eqnarray} $\gamma(s)=U=c\xi_{2}+c_{3}\xi_{3} $ is a constant vector. Moreover, $\int \xi_{2} ds$ and $\int \xi_{3}ds$ curves are helix curves with $U$ axis.\\
		\item[(iv)] If \begin{eqnarray}\nonumber
		- \frac{\overline{k}_{1}}{\overline{k}_{2}}=\tan s, \end{eqnarray} \begin{eqnarray}\nonumber \gamma(s)=c\sec s Y(s), \end{eqnarray} and $Y(s)\in S^{2}$. Hence, a type 2 rectifying-type curve $\gamma$ is obtained from a spherical curve (i.e. from $Y(s)$ spherical curve). This situation is consistent with conditions in Chen's papers.\\
	\end{itemize}
\end{theorem} 
\begin{proof}
	The proof can be given like the proof of Theorem (\ref{thm}).
\end{proof}
\begin{example}
	Let $\alpha(s)$ be any curve and  $\{T,N,B\}$ be Frenet frame of this curve. This Frenet frame is a Rotation minimizing frame on $\int N(s) ds$ and this frame formulas as follows:
	\begin{eqnarray}\nonumber
	\left[ 
	\begin{array}{c}
	N^{\prime }(s) \\ 
	B^{\prime }(s) \\ 
	T^{\prime }(s)	
	\end{array}
	\right] =\left[ 
	\begin{array}{ccc}
	0 & \tau & -\kappa \\
	-\tau & 0 & 0 \\ 
	\kappa& 0 & 0 
	\end{array}
	\right] \left[ 
	\begin{array}{c}
	N(s) \\ 
	B(s) \\ 
	T(s)
	\end{array}
	\right]. 
	\end{eqnarray} Here, $\overline{k}_{1}=\tau$, $\overline{k}_{2}=-\kappa$,  $\xi_{1}=N$, $\xi_{2}=B$ and $\xi_{3}=T$. Therefore, \begin{eqnarray}\nonumber
	\gamma=cB+c\frac{\tau}{\kappa}T, \quad c \in \mathbb{R},
	\end{eqnarray} is a type 2 rectifying-type curve.
	Then, for  type 2 rectifying-type curve $\gamma$, the following items holds:
	\begin{itemize}
		\item[(i)] $\gamma$ is a rectifying curve \cite{yay} (Modified Darboux).\\
		\item[(ii)] If \begin{eqnarray}\nonumber
		c \frac{\tau}{\kappa}=s+b, \quad c, b \in \mathbb{R}, \end{eqnarray} $\gamma(s)=(s+b)T+cB=\int T ds =\alpha(s)$ is a rectifying curve \cite{chen1,chen2,chen3,chen4}.\\
		\item[(iii)] If \ \begin{eqnarray}\nonumber
		-\frac{\tau}{\kappa}=constant, \end{eqnarray} $\gamma(s)=c_{3}T+cB, \quad c_{3}, \ c \in \mathbb{R} $\ \ is a constant vector. \ \ In addition,\ \ \ $U=cos \theta T+sin \theta B$. $\int T ds$ and $\int B ds$ curves are helix curves with $U$ axis.\\
		\item[(iv)] If \begin{eqnarray}\nonumber
		\frac{\tau}{\kappa}=\tan s, \end{eqnarray} \begin{eqnarray}\nonumber \gamma(s)=cB+c\tan s T=\frac{c}{\cos s}(\cos s B+\sin sT)=c\sec s Y(s), \end{eqnarray} and $Y(s)\in S^{2}$. Hence, a type 2 rectifying-type curve $\gamma$ is obtained from a spherical curve (i.e. from $Y(s)$ spherical curve). This situation is consistent with conditions in Chen's papers.\\
		\item[(v)] If $a\overline{k}_{1}+b\overline{k}_{2}+1=0, \quad  a,b \in \mathbb{R}$, $\int N ds$ is a spherical curve. Here, since  $\overline{k}_{1}=\tau$ and $\overline{k}_{2}=-\kappa$, $\alpha$ is a Bertrand curve.
	\end{itemize}
\end{example}
\section{Rectifying-Type Curves and Rotation Minimizing Frame in $\mathbb{R}^{4}$}
\par For $n=4$, the formulae of the Rotation minimizing frame $\{{\xi}_{1},{\xi}_{2},{\xi}_{3},{\xi}_{4}\}$ on $\int {\xi}_{1} ds$ is given in Equation (\ref{rmf1}).\\
\begin{definition}
	Let \ $\varphi(s)=f(s) \xi_{2}(s)+g(s) \xi_{3}(s)+h(s) \xi_{4}(s)$ \ be a rectifying-type curve. Then, 
		\begin{itemize}
		\item[(i)]If \ $g(s)=constant=a_{1}$ and $h(s)=constant=a_{2}$ are taken, \begin{eqnarray}\nonumber
		\varphi_{1}=f(s)\xi_{2}(s)+a_{1}\xi_{3}(s)+a_{2}\xi_{4}(s), \quad a_{1}, a_{2} \in \mathbb{R},
		\end{eqnarray} is defined as a type 1 rectifying-type curve.\\
		\item[(ii)] If  $f(s)=constant=b_{1}$ and $h(s)=constant=b_{2}$ is taken,\begin{eqnarray}\nonumber
		\varphi_{2}=b_{1}\xi_{2}(s)+g(s)\xi_{3}(s)+b_{2}\xi_{4}(s), \quad b_{1}, b_{2} \in \mathbb{R},
		\end{eqnarray} is defined as a type 2 rectifying-type curve.\\
		\item[(iii)]  If  $f(s)=constant=c_{1}$ and $g(s)=constant=c_{2}$ is taken,\begin{eqnarray}\nonumber
		\varphi_{3}=c_{1}\xi_{2}(s)+c_{2}\xi_{3}(s)+h(s)\xi_{4}(s), \quad c_{1}, c_{2} \in \mathbb{R},
		\end{eqnarray} is defined as a type 3 rectifying-type curve.\\
		
	\end{itemize}
\end{definition}
\begin{theorem} \label{thmm}
	Let $\{\xi_{1},\xi_{2},\xi_{3},\xi_{4}\}$ be a Rotation minimizing frame (i.e. Frenet-type frame) on $\alpha=\int \xi_{1}ds$ and 
	\begin{eqnarray}\nonumber
	\varphi_{1}=f(s)\xi_{2}(s)+a_{1}\xi_{3}(s)+a_{2}\xi_{4}(s), \quad a_{1}, a_{2} \in \mathbb{R},
	\end{eqnarray}  be a type 1 rectifying-type curve. If $\varphi_{1}'$ is a rectifying-type curve, the following items holds:
	\begin{itemize}
		\item[(i)] \begin{eqnarray}\nonumber
		\varphi_{1}=-(\frac{a_{1}\overline{k}_{2}+a_{2}\overline{k}_{3}}{\overline{k}_{1}})\xi_{2}+a_{1}\xi_{3}+a_{2}\xi_{4},\quad a_{1}, a_{2} \in \mathbb{R}, \end{eqnarray} is a rectifying curve.\\
		\item[(ii)] If \begin{eqnarray}\nonumber
		-(\frac{a_{1}\overline{k}_{2}+a_{2}\overline{k}_{3}}{\overline{k}_{1}})=s+b, \quad a_{1}, a_{2}\quad and\quad b \in \mathbb{R}, \end{eqnarray} $\varphi_{1}=\int \xi_{2} ds $ is a rectifying curve.\\
		\item[(iii)] If \begin{eqnarray}\nonumber
		-(\frac{a_{1}\overline{k}_{2}+a_{2}\overline{k}_{3}}{\overline{k}_{1}})=constant=c_{3}, \quad c_{3} \in \mathbb{R}, \end{eqnarray} $\varphi_{1}(s)=U=c_{3}\xi_{2}+a_{1}\xi_{3}+ a_{2}\xi_{4}$ is a constant vector. Moreover, $\int \xi_{2} ds$, $\int \xi_{3}ds$ and $\int \xi_{4}ds$ curves are helix curves with $U$ axis.\\
		\item[(iv)] If \begin{eqnarray}\nonumber
			-(\frac{a_{1}\overline{k}_{2}+a_{2}\overline{k}_{3}}{\overline{k}_{1}})=\tan s, \end{eqnarray} and $a_{1}^{2}+a_{2}^{2}=1$, \begin{eqnarray}\nonumber \varphi_{1}(s)=\sec s Y(s), \end{eqnarray} and $Y(s)\in S^{3}$. Hence, a type 1 rectifying-type curve $\varphi_{1}$ is obtained from a spherical curve (i.e. from $Y(s)$ spherical curve). This situation is consistent with conditions in Chen's papers.\\
	\end{itemize}
\end{theorem}
\begin{proof}
	\begin{itemize}
		\item[(i)] Since \begin{eqnarray}\nonumber
		\varphi_{1}'=f'(s)\xi_{2}(s)+(-f(s)\overline{k}_{1}- a_{1}\overline{k}_{2}-a_{2}\overline{k}_{3})\xi_{1},
		\end{eqnarray} $\varphi_{1}'$ is a rectifying-type curve if and only if \begin{eqnarray}\nonumber
		-f(s)\overline{k}_{1}- a_{1}\overline{k}_{2}-a_{2}\overline{k}_{3}=0, \end{eqnarray} and \begin{eqnarray}\nonumber
		f=-(\frac{a_{1}\overline{k}_{2}+a_{2}\overline{k}_{3}}{\overline{k}_{1}}). \end{eqnarray} \\
		Thus, for
		\begin{eqnarray}\nonumber
		\varphi_{1}=-(\frac{a_{1}\overline{k}_{2}+a_{2}\overline{k}_{3}}{\overline{k}_{1}})\xi_{2}+a_{1}\xi_{3}+a_{2}\xi_{4},\quad a_{1},a_{2} \in \mathbb{R}, \end{eqnarray} curve,  $T_{\varphi_{1}}=\xi_{2}$ and $N_{\varphi_{1}}=\xi_{1}$. Resultly, $\langle{\varphi_{1},\xi_{1}}\rangle=0$ or $\langle{\varphi_{1},N_{\varphi_{1}}}\rangle=0$ is obtained. $\varphi_{1}$ is a rectifying curve. 
		
		\item[(ii)] For \begin{eqnarray}\nonumber
		-(\frac{a_{1}\overline{k}_{2}+a_{2}\overline{k}_{3}}{\overline{k}_{1}})=s+b, \quad a_{1},a_{2}, b \in \mathbb{R}, \end{eqnarray} 
		\begin{eqnarray}\nonumber
		\varphi_{1}=(s+b)\xi_{2}+a_{1}\xi_{3}+a_{2}\xi_{4},\quad a_{1}, a_{1}, b \in \mathbb{R}, \end{eqnarray} rectifying curve overlap with $\int \xi_{2} ds $ curve.\\
		\item[(iii)] For \begin{eqnarray}\nonumber
		-(\frac{a_{1}\overline{k}_{2}+a_{2}\overline{k}_{3}}{\overline{k}_{1}})=constant=c_{3}, \quad a_{1}, a_{1}, c_{3} \in \mathbb{R}, \end{eqnarray} $U=c_{3}\xi_{2}+a_{1}\xi_{3}+a_{2}\xi_{4} $ is obtained. Since \begin{eqnarray}\nonumber
		-c_{3} \overline{k}_{1}-a_{1} \overline{k}_{2}-a_{2} \overline{k}_{3}=0, \end{eqnarray} \begin{eqnarray}\nonumber
		\frac{dU}{ds}=(-c_{3} \overline{k}_{1}-a_{1} \overline{k}_{2}-a_{2}\overline{k}_{3})\xi_{1}=0,\end{eqnarray} is found. Moreover, tangents of $\int \xi_{2} ds$, $\int \xi_{3} ds$ and $\int \xi_{4}ds$ curves makes the fixed angle with $U$ axis. Then, $\int \xi_{2} ds$, $\int \xi_{3} ds$ and $\int \xi_{4}ds$ curves are helix curves with $U$ axis.\\
		\item[(iv)] For \begin{eqnarray}\nonumber
		-(\frac{a_{1}\overline{k}_{2}+a_{2}\overline{k}_{3}}{\overline{k}_{1}})=\tan s, \end{eqnarray} and $a_{1}^{2}+a_{2}^{2}=1$, \begin{eqnarray}\nonumber \varphi_{1}(s)=\tan s\xi_{2}+a_{1}\xi_{3}+a_{2}\xi_{4}=\frac{1}{\cos s}(\sin s\xi_{2}+\cos s\xi_{3}+\cos s\xi_{4})=\sec s Y(s), \end{eqnarray} and $Y(s)\in S^{3}$.\\
	\end{itemize}
\end{proof}
\begin{proposition}
	If $a_{1}\overline{k}_{1}+a_{2}\overline{k}_{2}+a_{3}\overline{k}_{3}+1=0, \  a_{1}, a_{2} \ and  \ a_{3} \in \mathbb{R}$, $\int \xi_{1} ds$ is a spherical curve. Really, if \ $\int \xi_{1} ds$ is a spherical curve, $\int \xi_{1} ds$ is written as \begin{eqnarray*}
	\int \xi_{1} ds=a_{1}\xi_{2}+a_{2}\xi_{3}+a_{3}\xi_{4}.
	\end{eqnarray*}  If derivative of both sides of equality is taken,\begin{eqnarray*} \xi_{1}=(-a_{1}\overline{k}_{1}-a_{2}\overline{k}_{2}-a_{3}\overline{k}_{3})\xi_{1}, \end{eqnarray*}is obtained. Thus,\begin{eqnarray*} a_{1}\overline{k}_{1}+a_{2}\overline{k}_{2}+a_{3}\overline{k}_{3}+1=0, \ a_{1}, \ a_{2}, \ a_{3} \in \mathbb{R},\end{eqnarray*} $\int \xi_{1} ds$ is found. Also, \begin{eqnarray*} \Arrowvert \int \xi_{1} ds \lVert^{2}=r^{2}=a_{1}^{2}+a_{2}^{2}+a_{3}^{2}.\end{eqnarray*} $r$ is radius of a sphere. 
\end{proposition}
\begin{remark}
	For type 2 rectifying-type curves, \begin{eqnarray}\nonumber
	g=-(\frac{b_{1}\overline{k}_{1}+b_{2}\overline{k}_{3}}{\overline{k}_{2}}), \end{eqnarray} and for type 3 rectifying-type curves, \begin{eqnarray}\nonumber
	h=-(\frac{c_{1}\overline{k}_{1}+c_{2}\overline{k}_{2}}{\overline{k}_{3}}), \end{eqnarray} are obtained. The results given for the type 1 rectifying-type curves in Theorem (\ref{thmm}) can easily give in other types.
\end{remark} 
Now, we can generalize. For n, the formulae of the Rotation minimizing frame $\{{\xi}_{1},{\xi}_{2},...,{\xi}_{n}\}$ on $\int {\xi}_{1} ds$ is given in Equation (\ref{rmf2}).\\
\begin{definition}
	Let \ $\psi(s)=f_{1}(s) \xi_{2}(s)+f_{2}(s) \xi_{3}(s)+...+f_{n-1}(s) \xi_{n}(s)$ \ be a rectifying-type curve. Then, 
	\begin{itemize}
		\item[(a)]If \ $f_{2}(s)=constant=a_{1}$, ..., $f_{n-1}(s)=constant=a_{n-2}$ are taken, \begin{eqnarray}\nonumber
		\psi_{1}=f_{1}(s)\xi_{2}(s)+a_{1}\xi_{3}(s)+...+a_{n-2}\xi_{n}(s), \quad a_{1},..., a_{n-2} \in \mathbb{R},
		\end{eqnarray} is defined as a type 1 rectifying-type curve.\\
		\item[(b)] If  $f_{1}(s)=constant=b_{1}$, $f_{3}(s)=constant=b_{2}$,..., $f_{n-1}(s)=constant=b_{n-2}$ is taken,
		\begin{eqnarray}\nonumber
        \psi_{2}=b_{1}\xi_{2}(s)+f_{2}(s)\xi_{3}(s)+b_{2}\xi_{4}(s)+...+ b_{n-2}\xi_{n}(s), \quad b_{1},..., b_{n-2} \in \mathbb{R},
		\end{eqnarray} is defined as a type 2 rectifying-type curve.\\
		...\\
		\item[(c)]  If  $f_{1}(s)=constant=c_{1}$,..., $f_{n-2}(s)=constant=c_{n-3}$ is taken,\begin{eqnarray}\nonumber
		\psi_{n}=c_{1}\xi_{2}(s)+c_{2}\xi_{3}(s)+...+f_{n-1}(s)\xi_{n}(s), \quad c_{1},..., c_{n-2} \in \mathbb{R},
		\end{eqnarray} is defined as a type (n-1) rectifying-type curve.\\
		
	\end{itemize}
\end{definition}
\begin{theorem}\label{thmmm}
	Let $\{\xi_{1},\xi_{2},...,\xi_{n}\}$ be a Rotation minimizing frame (i.e. Frenet-type frame) on $\alpha=\int \xi_{1}ds$ and 
	\begin{eqnarray}\nonumber
	\psi_{1}=f_{1}(s)\xi_{2}(s)+a_{1}\xi_{3}(s)+...+a_{n-2}\xi_{n}(s), \quad a_{1},..., a_{n-2} \in \mathbb{R},
	\end{eqnarray}  be a type 1 rectifying-type curve. If $\psi_{1}'$ is a rectifying-type curve, the following items holds:
	\begin{itemize}
		\item[(i)] \begin{eqnarray}\nonumber
			\psi_{1}=-(\frac{a_{1}\overline{k}_{2}+...+a_{n-2}\overline{k}_{n-1}}{\overline{k}_{1}})\xi_{2}+a_{1}\xi_{3}+...+a_{n-2}\xi_{n},\quad a_{1},..., a_{n-2} \in \mathbb{R}, \end{eqnarray} is a rectifying curve.\\
		\item[(ii)] If \begin{eqnarray}\nonumber
		-(\frac{a_{1}\overline{k}_{2}+...+a_{n-2}\overline{k}_{n-1}}{\overline{k}_{1}})=s+b, \quad a_{1},..., a_{n-2}\quad and\quad b \in \mathbb{R}, \end{eqnarray} $\psi_{1}=\int \xi_{2} ds $ is a rectifying curve.\\
		\item[(iii)] If \begin{eqnarray}\nonumber
		-(\frac{a_{1}\overline{k}_{2}+...+a_{n-2}\overline{k}_{n-1}}{\overline{k}_{1}})=constant=c_{3}, \quad c_{3} \in \mathbb{R}, \end{eqnarray} $\psi_{1}(s)=U=c_{3}\xi_{2}+a_{1}\xi_{3}+...+ a_{n-2}\xi_{n}$ is a constant vector. Moreover, $\int \xi_{2} ds$, $\int \xi_{3}ds$,..., $\int \xi_{n}ds$ curves are helix curves with $U$ axis.\\
		\item[(iv)] If \begin{eqnarray}\nonumber
		-(\frac{a_{1}\overline{k}_{2}+...+a_{n-2}\overline{k}_{n-1}}{\overline{k}_{1}})=\tan s, \end{eqnarray} and $a_{1}^{2}+...+a_{n-2}^{2}=1$, \begin{eqnarray}\nonumber \psi_{1}(s)=\sec s Y(s), \end{eqnarray} and $Y(s)\in S^{n-1}$. Hence, a type 1 rectifying-type curve $\psi_{1}$ is obtained from a spherical curve (i.e. from $ Y(s)$ spherical curve). This situation is consistent with conditions in Chen's papers.\\
	\end{itemize}
\end{theorem}
\begin{proof}
	The proof can be given like the proof of Theorem (\ref{thmm}).
\end{proof}
\begin{proposition}
	If $a_{1}\overline{k}_{1}+...+a_{n-1}\overline{k}_{n-1}+1=0, \  a_{1},..., a_{n-1} \in \mathbb{R}$, $\int \xi_{1} ds$ is a spherical curve. Really, if \ $\int \xi_{1} ds$ is a spherical curve, $\int \xi_{1} ds$ is written as \begin{eqnarray*}
		\int \xi_{1} ds=a_{1}\xi_{2}+a_{2}\xi_{3}+...+a_{n-1}\xi_{n}.
	\end{eqnarray*}  If derivative of both sides of equality is taken,\begin{eqnarray*} \xi_{1}=(-a_{1}\overline{k}_{1}-a_{2}\overline{k}_{2}-...-a_{n-1}\overline{k}_{n-1})\xi_{1}, \end{eqnarray*}is obtained. Thus,\begin{eqnarray*} a_{1}\overline{k}_{1}+...+a_{n-1}\overline{k}_{n-1}+1=0, \  a_{1},..., a_{n-1} \in \mathbb{R}, \end{eqnarray*} $\int \xi_{1} ds$ is found. Also, \begin{eqnarray*} \Arrowvert \int \xi_{1} ds \lVert^{2}=r^{2}=a_{1}^{2}+...+a_{n-1}^{2}.\end{eqnarray*} $r$ is radius of a sphere. 
\end{proposition}
\begin{remark}
	For type 2 rectifying-type curves, \begin{eqnarray}\nonumber
	f_{2}=-(\frac{b_{1}\overline{k}_{1}+...+b_{n-2}\overline{k}_{n-1}}{\overline{k}_{2}}), \end{eqnarray} ... \\ and for type $(n-1)$ rectifying-type curves, \begin{eqnarray}\nonumber
	f_{n-1}=-(\frac{c_{1}\overline{k}_{1}+...+c_{n-2}\overline{k}_{n-2}}{\overline{k}_{n-1}}), \end{eqnarray} are obtained. The results given for the type 1 rectifying-type curves in Theorem (\ref{thmmm}) can easily give in other types.
\end{remark} 
\begin{example}
     Let the following helix curve. \\  
     \begin{eqnarray}\nonumber
	\mu=\mu(s)=(24 \cos \frac{s}{25}, 24 \sin \frac{s}{25}, \frac{7s}{25}).
	\end{eqnarray}$\{T,N_{1},N_{2}\}$ is a Rotation minimizing frame on $\mu=\mu(s)$ helix curve (this frame is also Bishop frame).
	Find type-1 and type-2 rectifying-type curves for this curve.\\
	
		\begin{displaymath}
		\left\{ \begin{array}{l}
		\textrm{{\Large $T$=(-$\frac{24}{25}$ $\sin$ $\frac{s}{25}$, $\frac{24}{25}$ $\cos$ $\frac{s}{25}$, $\frac{7s}{25}$)}},\\ \\
		\textrm{{\Large $N$=(-$\cos$$\frac{s}{25}$, -$\sin$$\frac{s}{25}$, 0)}},\\
		\\
		\textrm{{\Large $B$=$T\wedge N$= ($\frac{7}{25}$ $\sin$ $\frac{s}{25}$, -$\frac{7}{25}$ $\cos$ $\frac{s}{25}$, $\frac{24}{25}$)}},\\
		\\
		\textrm{{\LARGE $N_{1}$=(-$\cos$$\frac{7s}{625}$ $\cos$ $\frac{s}{25}$$-$$\frac{7}{25}$$\sin$$\frac{7s}{625}$ $\sin$ $\frac{s}{25}$,
	}}\\
		\\
		\textrm{{\LARGE \ \ \ \ \ \ \ -$\cos$$\frac{7s}{625}$ $\sin$ $\frac{s}{25}$$+$$\frac{7}{25}$$\cos$$\frac{s}{625}$ $\sin$ $\frac{7s}{625}$, -$\frac{24}{25}$ $\sin$ $\frac{7s}{625}$)}},\\
		\\
		\textrm{{\LARGE $N_{2}$=(-$\sin$$\frac{7s}{625}$ $\cos$ $\frac{s}{25}$$+$$\frac{7}{25}$$\cos$$\frac{7s}{625}$ $\sin$ $\frac{s}{25}$,
			}}\\
			\\
			\textrm{{\LARGE \ \ \ \ \ \ \ -$\sin$$\frac{7s}{625}$ $\sin$ $\frac{s}{25}$$-$$\frac{7}{25}$$\cos$$\frac{7s}{625}$ $\cos$ $\frac{s}{25}$, $\frac{24}{25}$ $\cos$ $\frac{7s}{625}$)}}.\\
				\end{array} \right.
		\end{displaymath} 
		and
		\begin{displaymath}
		\left\{ \begin{array}{l}
		\textrm{{\Large $\kappa$=$\frac{24}{625}$}},\\ \\
		\textrm{{\Large $\tau$=$\frac{7}{625}$}},\\
		\\
		\textrm{{\Large $\vartheta(s)$=$\int  \tau ds$}},\\
		\\
		\textrm{{\Large $k_{1}$=$\frac{24}{625}$$\cos$$\vartheta(s)$}},\\
			\\
			\textrm{{\Large $k_{2}$=$\frac{24}{625}$$\sin$$\vartheta(s)$}}.\\
			
				\end{array} \right.
				\end{displaymath} 
	Type-1 and type-2 rectifying-type curves for this curve are as follows:			
					\begin{displaymath}
					\left\{ \begin{array}{l}
					\textrm{{\Large $\beta_{1}$=-c $\frac{k_{2}}{k_{1}}$ $N_{1}$$+$c $N_{2}$}},\\ \\
					\textrm{{\Large $\beta_{2}$=c $N_{1}$$+$c $\frac{k_{2}}{k_{1}}$ $N_{2}$}}.\\
					\end{array} \right.
					\end{displaymath} 
	Here, if $c = 1$ is taken and if put in place $k_{1}$, $k_{2}$, $N_{1}$, $N_{2}$, \\	
	\begin{displaymath}
	\left\{ \begin{array}{l}
	\textrm{{\Large $\beta_{1}$=- $\frac{k_{2}}{k_{1}}$ $N_{1}$$+$$N_{2}$}},\\ \\
	\textrm{{\Large $\beta_{2}$=$N_{1}$$+$ $\frac{k_{2}}{k_{1}}$ $N_{2}$}}.\\
	\end{array} \right.
	\end{displaymath}
		\begin{displaymath}
		\left\{ \begin{array}{l}
		\textrm{{\Large $\beta_{1}$=($\tan$$\frac{7s}{625}$ $\cos$$\frac{7s}{625}$ $\cos$ $\frac{s}{25}$$+$$\frac{7}{25}$$\tan$$\frac{7s}{625}$$\sin$$\frac{7s}{625}$ $\sin$ $\frac{s}{25}$$-$$\sin$$\frac{7s}{625}$ $\cos$ $\frac{s}{25}$$+$$\frac{7}{25}$$\cos$$\frac{7s}{625}$ $\sin$ $\frac{s}{25}$,}}\\ \\
		\textrm{{\Large \ \ \ \ \ \ \ $\tan$$\frac{7s}{625}$$\cos$$\frac{7s}{625}$ $\sin$ $\frac{s}{25}$$-$$\frac{7}{25}$$\tan$$\frac{7s}{625}$$\cos$$\frac{s}{25}$ $\sin$ $\frac{7s}{625}$$-$$\sin$$\frac{7s}{625}$ $\sin$ $\frac{s}{25}$$-$$\frac{7}{25}$$\cos$$\frac{7s}{625}$ $\cos$ $\frac{s}{25}$},}\\ \\
		\textrm{{\Large \ \ \ \ \ \ \ $\frac{24}{25}$$\tan$$\frac{7s}{625}$ $\sin$ $\frac{7s}{625}$$+$$\frac{24}{25}$ $\cos$ $\frac{7s}{625}$),}}\\ 
		\\
			\textrm{{\Large $\beta_{2}$=(-$\cos$$\frac{7s}{625}$ $\cos$ $\frac{s}{25}$$-$$\frac{7}{25}$$\sin$$\frac{7s}{625}$$\sin$ $\frac{s}{25}$$-$$\cot$$\frac{7s}{625}$$\sin$$\frac{7s}{625}$ $\cos$ $\frac{s}{25}$$+$$\frac{7}{25}$$\cot$$\frac{7s}{625}$$\cos$$\frac{7s}{625}$ $\sin$ $\frac{s}{25}$,}}\\ \\
			\textrm{{\Large \ \ \ \ \ \ \ -$\cos$$\frac{7s}{625}$ $\sin$ $\frac{s}{25}$$+$$\frac{7}{25}$$\cos$$\frac{s}{25}$ $\sin$ $\frac{7s}{625}$$-$$\cot$$\frac{7s}{625}$$\sin$$\frac{7s}{625}$ $\sin$ $\frac{s}{25}$$-$$\frac{7}{25}$$\cot$$\frac{7s}{625}$$\cos$$\frac{7s}{625}$ $\cos$ $\frac{s}{25}$},}\\ \\
			\textrm{{\Large \ \ \ \ \ \ \ -$\frac{24}{25}$ $\sin$ $\frac{7s}{625}$$+$$\frac{24}{25}$$\cot$$\frac{7s}{625}$ $\cos$ $\frac{7s}{625}$),}}\\
		
				\end{array} \right.
				\end{displaymath} 
\end{example}
	\begin{figure}[tbph]
		\centering
		\includegraphics[width=0.3\linewidth, height=0.2\textheight]{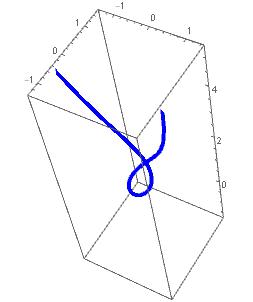}
		\caption[]{\centering $\beta_{1}$, Type-1 rectifying-type curves}
		\label{fig:fig1}
	\end{figure}
	\begin{figure}[tbph]
		\centering
		\includegraphics[width=0.4\linewidth, height=0.3\textheight]{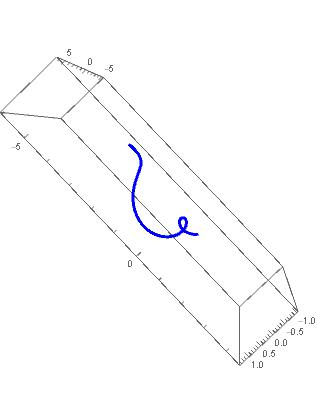}
		\caption[]{\centering $\beta_{2}$, Type-2 rectifying-type curves}
		\label{fig:fig2}
	\end{figure}
	\newpage
\section*{Conclusion}
\indent In this paper, we have given an implementation of Rotation minimizing frames (RMF) using Myller configuration. Here, we generalized easily rectifying-type curves n dimensional space $\mathbb{R}^{n}$. Also, we have seen that special curves are characterized very easily with the help of this frame and the hypothesis \lq \lq the derivative of the rectifying type curves  is of the rectifying-type curves\rq \rq\  makes an important contribution to the classification of special curves. In later studies, this study will be discussed in the Minkowski space. 
 \section*{Acknowledgement}
The first author would like to thank Tubitak-Bidep for their financial supports during her PhD studies.

\end{document}